\documentclass{amsart}

\usepackage{amsmath}%
\usepackage{amsfonts}%
\usepackage{amssymb}%
\usepackage{graphicx}
\usepackage{subfigure}
\usepackage{color}
\usepackage{dsfont}
\usepackage{bbm}

\newcommand{\pa}{\partial}
\def\R{{\mathbb R}}
\def\N{{\mathbb N}}
\def\Vens{{\bf \mathbb V}}
\def\rn{{{\mathbb R}^n}}
\newcommand{\tH}{{\widetilde H}}
\newcommand{\tHs}{{\widetilde H}^s(\Omega)}
\DeclareMathOperator{\supp}{supp}

\theoremstyle{plain}
\newtheorem{cor}{Corollary}
\newtheorem{dfn}{Definition}

\newtheorem{pro}{Proposition}
\newtheorem{rmk}{Remark}

\numberwithin{equation}{section}
\numberwithin{table}{section}
\numberwithin{figure}{section}
\numberwithin{thm}{section}
\numberwithin{cor}{section}
\numberwithin{dfn}{section}
\numberwithin{lem}{section}
\numberwithin{pro}{section}
\numberwithin{rmk}{section}

\begin{document}
\title{On the convergence in $H^1$-norm for the fractional Laplacian}
\author[J.P.~Borthagaray]{Juan Pablo~Borthagaray}
\address[J.P.~Borthagaray]{Department of Mathematics, University of Maryland, College Park, MD 20742, USA}
\email{jpb@math.umd.edu}
\thanks{JPB has been supported in part by NSF grant DMS-1411808}

\author[P.~Ciarlet~Jr.]{Patrick~Ciarlet~Jr.}
\address[P.~Ciarlet~Jr.]{POEMS, ENSTA ParisTech, CNRS, INRIA, Universit\'{e} Paris-Saclay\\ 828 Bd des Mar\'{e}chaux, 91762 Palaiseau Cedex, France}
\email{patrick.ciarlet@ensta-paristech.fr}

\date{\today}
\begin{abstract}We consider the numerical solution of the fractional Laplacian of index $s\in(1/2,1)$ in a bounded domain $\Omega$ with homogeneous boundary conditions. Its solution a priori belongs to the fractional order Sobolev space ${\widetilde H}^s(\Omega)$. For the Dirichlet problem and under suitable assumptions on the data, it can be shown that its solution is also in $H^1(\Omega)$. In this case, if one uses the standard Lagrange finite element to discretize the problem, then both the exact and the computed solution belong to $H^1(\Omega)$. A natural question is then whether one can obtain error estimates in $H^1(\Omega)$-norm, in addition to the classical ones that can be derived in the ${\widetilde H}^s(\Omega)$ energy norm. We address this issue, and in particular we derive error estimates for the Lagrange finite element solutions on both quasi-uniform and graded meshes.
\end{abstract}
\maketitle

\section{Introduction}

Let $\Omega$ be a bounded Lipschitz domain in $\R^n$ satisfying the exterior ball condition. In this paper, we study the fractional Laplace equation of index $s\in(1/2,1)$ 
\begin{equation} \label{eq:dirichlet}
\left\lbrace
  \begin{array}{rl}
      (-\Delta)^s u = f & \mbox{ in }\Omega, \\
      u =  0 & \mbox{ in }\Omega^c = \R^n \setminus \Omega . \\
      \end{array}
    \right.
\end{equation}
We call $f$ the right-hand side, which is a priori in $L^\infty(\Omega)$. The operator $(-\Delta)^s$ is called the fractional Laplacian of order $s$, and it is one of the most prominent nonlocal {operators}. It is ubiquitous in the modeling  of complex physical, biological and social phenomena that span vastly different length scales \cite{MeKl00}.

There is a clear way to define the fractional Laplacian of order $s$ for functions defined over $\rn$. Indeed, it is the pseudo-differential operator with symbol $|\xi|^{2s}$; given a function $u$ in the Schwartz class, set
\begin{equation*}\label{eq:Fourier}
(-\Delta)^s u := \mathcal{F}^{-1} \left( |\xi|^{2s} \mathcal{F} u \right) ,
\end{equation*}
where $\mathcal{F}$ denotes the Fourier transform. Equivalently, the fractional Laplacian can be defined by means of the following pointwise formula (see \cite[Section 1.1]{La72} and \cite[Proposition 3.3]{DNPaVa12})
\begin{equation}\label{eq:frac_lap}
(-\Delta)^s u(x) = C(n,s) \mbox{ p.v.} \int_\rn \frac{u(x)-u(y)}{|x-y|^{n+2s}} \, dy, \quad C(n,s) = \frac{2^{2s} s \Gamma(s+\frac{n}{2})}{\pi^{n/2}\Gamma(1-s)},
\end{equation}
where \mbox{p.v.} stands for the Cauchy principal value and $C(n,s)$ is a normalization constant. Identity \eqref{eq:frac_lap} makes evident the non-local structure of the fractional Laplacian. In the theory of stochastic processes, this operator appears as the infinitesimal generator of a $2s$-stable L\'evy process \cite{Be96}. We refer the reader to \cite{Kw17} for further characterizations of the fractional Laplacian over $\rn$.

There is not a unique mode to consistently extend the definition of the fractional Laplacian over a  bounded domain $\Omega$; see \cite{BBNOS18,DuWaZh17,Letal18} for a comparison of the different definitions and related numerical methods. In this work, we consider the {\em integral} fractional Laplacian, which is defined as follows. Given $u \in C^\infty_0(\Omega)$, we first consider the zero-extension of $u$ onto $\Omega^c$ and then use definition \eqref{eq:frac_lap}. 
This definition maintains the probabilistic interpretation of the fractional Laplacian defined over $\rn$, that is, as the generator of a random walk in $\Omega$ with arbitrarily long jumps, where particles are killed upon reaching $\Omega^c$ \cite[Chapter 2]{BuVa16}. 

For this operator, we analyze direct discretizations for problem \eqref{eq:dirichlet} using linear Lagrangian finite elements. Under the assumption that $f \in [\widetilde{H}^s(\Omega)]^*$, which is clearly true as long as we consider $f \in L^\infty(\Omega)$, it follows immediately that the solution $u$ to \eqref{eq:dirichlet} belongs to $\widetilde{H}^s(\Omega)$ (cf. \S\ref{Sobolev} for a definition of these spaces). Obviously, computing finite element solutions to \eqref{eq:dirichlet} is nothing more than projecting $u$ over the discrete spaces with respect to the ${\widetilde H}^s(\Omega)$ energy norm. Therefore, it is natural to derive convergence rates for such a method in the ${\widetilde H}^s(\Omega)$-norm \cite{AcBo17,AcBoHe18,AiGl17,DEGu13}.
Additionally, convergence rates in the $L^2(\Omega)$-norm can be obtained by performing a duality argument \textit{\`{a} la} Aubin-Nitsche \cite{BoDM16}. In \cite{BoLePa17}, error estimates in the $L^2$ norm are derived for a related finite element discretization, based on a Dumford-Taylor representation formula for the weak form of the fractional Laplacian.

In the case $s \in (1/2,1)$, under additional assumptions on the right-hand side, it can be proven that $u \in H^1(\Omega)$. Since the discrete functions also belong to $H^1(\Omega)$ (in fact, the discrete spaces are contained in $\cap_{\epsilon > 0} H^{3/2-\epsilon}(\Omega)$), a natural question is whether one can obtain error estimates in $H^1(\Omega)$ norm. 
The goal of this work is to address such a question. In particular, we derive error estimates for the Lagrange finite element solutions on both quasi-uniform and graded meshes.

Let us outline the contents of the paper. In Section~\ref{settings}, we recall some useful results regarding the problem to be solved and the regularity of its solution. More precisely, the fractional Laplacian defined over $\rn$ can be extended by density to the Sobolev space $\tHs$, see \S\ref{Sobolev} for a definition of this space. Then, one can build an equivalent variational form (\S\ref{VF_Htilde_s}). Under suitable assumptions on the data, it can be shown that its solution also belongs to $H^1(\Omega)$; the regularity results are recalled in \S\ref{regularity}. To solve the problem numerically, we choose the standard Lagrange finite element to define a conforming discretization (\S\ref{conforming_approx-1} and \S\ref{conforming_approx-2}). As pointed out before, both the exact and the computed solution belong to $H^1(\Omega)$.

We address the issue of convergence in $H^1(\Omega)$ norm, first on quasi-uniform meshes (Section~\ref{quasi-uniform}), and then on graded meshes (Section~\ref{graded}).
On quasi-uniform meshes, a use of mostly classical estimates (interpolation error, inverse inequality, ...)  allows us to conclude that convergence in $H^1(\Omega)$ norm holds, with a rate in the order of $h^{s-1/2}$ (up to a $|\log h|$ factor),  where $h$ is the mesh-size. On the other hand, it is well-known that choosing graded meshes can improve the convergence rate in problems with boundary layers.
For instance, for the same type of discretizations as the ones considered in this paper, this procedure allows to recover a rate in the order of $h$ (up to a $|\log h|$ factor) in the energy norm \cite{AcBo17}. In particular, the grading must be chosen carefully in order to keep an optimal convergence rate in terms of the dimension of the discrete finite element space. Also, one has to build estimates with respect to weighted Sobolev norms. Section~\ref{graded} is devoted to this task.

In Section~\ref{num_exp} we present some numerical experiments to highlight the results, and in particular how the predicted convergence rate is recovered numerically. Finally, in Section~\ref{sec:conclusion} we comment on the results in this manuscript and discuss possible extensions of this work.

\section{Settings and preliminaries}\label{settings}

\subsection{Sobolev spaces} \label{Sobolev}

Given $s \in (0,1)$ and $\Lambda \subset \rn$ (with the possibility that $\Lambda = \rn$), we define the Sobolev space $H^s(\Lambda)$ as
\begin{equation}\label{defn-Hs}
H^s(\Lambda) = \left\{ v \in L^2 (\Lambda) \colon |v|_{H^s(\Lambda)} < \infty \right\},\quad
\mbox{where} \quad|v|_{H^s(\Lambda)} = (v,v)_{H^s(\Lambda)}^{1/2},
\end{equation}
with
\begin{equation} \label{eq:def_innerprod_Lambda}
(v, w)_{H^s(\Lambda)} = \frac{C(n,s)}2 \iint_{\Lambda \times \Lambda} \frac{(v(x)-v(y))(w(x) - w(y))}{|x-y|^{n+2s}} \, dx dy
\end{equation}
and $C(n,s)$ is defined as in \eqref{eq:frac_lap}.

Sobolev spaces of non-integer order greater than one are defined as follows. Given $k \in \mathbb{N}$, then
\[
H^{k+s} (\Lambda) = \left\{ v \in H^k(\Lambda) \colon \pa^\beta v \in H^s(\Lambda) \ \forall \beta \mbox{ s.t. } |\beta| = k \right\} ,
\]
equipped with the norm
\[
\| v \|_{H^{k+s} (\Lambda)} =  \Big( \| v \|_{H^k(\Lambda)}^2 + \sum_{|\beta| = k } | \pa^\beta v |_{H^s(\Lambda)}^2 \Big)^{1/2}.
 \]
For a Lipschitz bounded domain $\Omega \subset \rn$, we denote by $\tHs$ the space defined by
\begin{equation*} \label{eq:def_tHs}
  \tHs = \left\{ v \in H^s(\rn) \mbox{ s.t. } \supp(v) \subset \overline \Omega \right\}.
\end{equation*}

We point out that, on $\tHs$, the natural inner product is equivalent to
\begin{eqnarray}
(v, w)_{H^s(\rn)}&=& \frac{C(n,s)}2 \iint_{\rn \times \rn} \frac{(v(x)-v(y))(w(x) - w(y))}{|x-y|^{n+2s}} dx dy,\label{eq:def_innerprod} \\
 \|v\|_{\tHs} &=& (v,v)_{H^s(\rn)}^{1/2}.\nonumber
\end{eqnarray}
because of the Poincar\'e inequality
\[
\| v \|_{L^2(\Omega)} \lesssim | v |_{H^s(\rn)}, \quad \forall v \in \tHs.
\]

It is well-known that smooth functions are dense in $H^s(\Omega)$. Another way to regard ``zero-trace'' functions on $\Omega$ is to take the closure of $C^\infty_0(\Omega)$ with respect to the $H^s(\Omega)$-norm. This gives rise to the space
\[
H^s_0(\Omega) = \overline{C^\infty_0(\Omega)}^{\| \cdot \|_{H^s{(\Omega)}}}.
\]

For $s \in (0,1)$, the aforementioned Sobolev spaces on $\Omega$ are related by
\[ \begin{aligned}
& \tHs = H^s_0(\Omega) = H^s(\Omega) & \mbox{ if } s \in (0,1/2), \\
& \widetilde{H}^{1/2}(\Omega) \subsetneq H^{1/2}_0(\Omega) = H^{1/2}(\Omega) & \mbox{ if } s =1/2, \\
& \tHs = H^s_0(\Omega)  \subsetneq H^s(\Omega) & \mbox{ if } s \in (1/2,1).
\end{aligned} \]

\begin{rmk}[Interpolation spaces]
Because $\Omega$ is a Lipschitz domain, we can also characterize fractional Sobolev spaces over $\Omega$ as real interpolation spaces. Namely, 
\[
H^s(\Omega) = [L^2(\Omega), H^1(\Omega)]_s, \qquad \tHs = [L^2(\Omega), H_0^1(\Omega)]_s, 
\]
and the norms induced by this characterization are equivalent to \eqref{eq:def_innerprod_Lambda} and \eqref{eq:def_innerprod}, respectively.
\end{rmk}

\begin{rmk}[Exact interpolation scales] A subsequent question is whether the fractional Sobolev spaces $\widetilde H^s(\Omega)$ are \emph{exact} interpolation spaces in the sense of \cite{BeLo76}, that is, whether the fractional-order norms coincide with the norms inherited by interpolation. We point out that, in general, this is not the case: the set 
$ \{ H^s(\Lambda) \colon s \in \R \}$ normed by (\ref{defn-Hs})-(\ref{eq:def_innerprod_Lambda})
is not an exact interpolation scale \cite{ChHeMo15}. On the other hand, for a Lipschitz domain $\Omega$, the equivalence constants depend on the continuity modulus of certain extension operators. This result is also valid for the $\widetilde H^s(\Omega)$ spaces by duality with $H^{-s}(\Omega) = [\tHs]^*$ spaces. For further details, we refer the reader to Section 4 in \cite{ChHeMo15}, specifically to Lemma 4.2, Corollary 4.9 and Lemma 4.13 therein.
\end{rmk}

\begin{rmk}[Normalization constant] 
The normalization constant $C(n,s)$ in the definition of fractional Sobolev spaces compensates the singular behavior of the Gagliardo seminorms  as $s$ approaches $0$ and $1$. Indeed, it satisfies
\[
C(n,s) \approx s(1-s) \qquad \mbox{as } s \to 0, 1.
\]

In the limit $s \to 0$, the presence of $C(n,s)$ ensures that (see \cite[Theorem 3]{MaSh02})
\[
\lim_{s \to 0} |v|_{H^s(\rn)} = \| v \|_{L^2(\rn)}, \quad \forall v \in H^\sigma_0(\rn)  \mbox{ for some } \sigma >0.
\] 
In particular, we have the limit
\[
\lim_{s \to 0} \|v\|_{\widetilde H^s(\Omega)} = \| v \|_{L^2(\Omega)}, \quad \forall v \in \widetilde H^\sigma(\Omega) \mbox{ for some } \sigma >0.
\]

Similarly, in the limit $s\to 1$, the following estimate holds: given $v \in L^2(\Omega)$, if $\lim_{s \to 1}  |v|_{H^s(\Omega)}$ exists and it is finite, then $v \in H^1(\Omega)$ and
\begin{equation} \label{eq:lim_s_to_1} \begin{split}
\lim_{s \to 1}  |v|_{H^s(\Omega)} = | v |_{H^1(\Omega)}.
\end{split} \end{equation}
We refer the reader to \cite{BoBrMi01} for a proof. Although Corollary 2 in that work is mainly concerned with the need of a factor of the order $\sqrt{1-s}$ to correct the scaling of the Gagliardo seminorms as $s \to 1$, we point out that a direct calculation shows that identity \eqref{eq:lim_s_to_1} holds.  
\end{rmk}

Regarding problem \eqref{eq:dirichlet}, it is known that, independently of the smoothness of the right-hand side $f$, solutions exhibit reduced regularity near the boundary of the domain. More precisely, denoting by $\delta(x)$ the distance from $x\in\Omega$ to $\pa\Omega$, solutions to the fractional Dirichlet problem are of the form \cite[formulas (7.7)--(7.12)]{Gr15}
\begin{equation}\label{eq:boundary-grubb}
u(x) \approx \delta(x)^s + v(x),
\end{equation}
with $v$ smooth. Thus, a natural approach to characterize the behavior of the solution to \eqref{eq:dirichlet} near the boundary is to introduce weighted Sobolev spaces, where the weight is a power of the distance to the boundary.

For a non-negative integer $k$ and $\alpha \in \R$, we consider the norm
\begin{equation}
\label{eq:defofWnorm}
  \| v \|_{H^k_\alpha(\Omega)}^2 = \int_\Omega \left( |v(x)|^2 + \sum_{|\beta| \leq k} |\pa^\beta v(x)|^2 \right) \delta(x)^{2\alpha} dx,
\end{equation}
and define $H^k_\alpha(\Omega)$ and $\tH^k_\alpha(\Omega)$ as the closures of $C^\infty(\Omega)$ and $C_0^\infty(\Omega)$, respectively, with respect to the norm \eqref{eq:defofWnorm}.

Next, we define weighted Sobolev spaces of non-integer order and their zero-extension counterparts.

\begin{dfn}[Weighted fractional Sobolev spaces] \label{def:weighted}
Let $\ell$ be a non-integer and positive real number, and let $\alpha \in \R$. Take $k \in \N \cup\{0\}$ and $\sigma \in (0,1)$ to be the unique numbers such that $\ell = k + \sigma$. We set
\begin{equation*}
  H^\ell_\alpha (\Omega) = \left\{ v \in H^k_\alpha(\Omega) \colon  \ | \pa^\beta v |_{H^{\sigma}_\alpha (\Omega)} < \infty,  \ \forall \beta \mbox{ s.t. } |\beta| = k  \right\} ,
\end{equation*}
where 
\[
| v |^2_{H^{\sigma}_\alpha (\Omega)} = \iint_{\Omega\times\Omega} \frac{|v(x)-v(y)|^2}{|x-y|^{n+2\sigma}} \, \delta(x,y)^{2\alpha} dx \,  dy,
\]
and
\[
\delta(x,y) = \min\{ \delta(x), \delta(y) \}.
\]

We equip this space with the norm
\[
  \| v \|_{H^\ell_\alpha (\Omega)}^2 = \| v \|_{H_\alpha^k (\Omega)}^2 + \sum_{|\beta| = k } | \pa^\beta v |^2_{H^{\sigma}_\alpha (\Omega)} .
\]
Similarly, we define zero-extension weighted Sobolev spaces by
\begin{equation*}
  \tH^\ell_\alpha (\Omega) = \left\{ v \in \tH^k_\alpha(\Omega) \colon \ | \pa^\beta v |_{H^{\sigma}_\alpha (\rn)} < \infty,  \ \forall \beta \mbox{ s.t. } |\beta| = k   \right\} ,
\end{equation*}
equipped with the norm
\[
  \| v \|_{\tH^\ell_\alpha (\Omega)}^2 = \| v \|_{H_\alpha^k (\Omega)}^2 + \sum_{|\beta| = k } | \pa^\beta v |^2_{H^{\sigma}_\alpha (\rn)} .
\]
\end{dfn}

Throughout this paper we make use of the $H^\ell_\alpha(\omega)$ and $\tH^\ell_\alpha(\omega)$ norms and seminorms, where $\omega$ is a Lipschitz subdomain of $\Omega$. We point out that, in such a case, the weight $\delta$ still refers to the distance to $\pa\Omega$.

\subsection{Weak formulation} \label{VF_Htilde_s}

We denote the duality pairing between $\tHs$ and its dual $H^{-s}(\Omega)$ by $\langle \cdot , \cdot \rangle$. The fractional Laplacian of index $s$ is an operator of order $2s$; therefore, $(-\Delta)^s v \in H^{-s}(\Omega)$ whenever $v \in \tHs$. The following integration by parts formula is a direct consequence of definitions \eqref{eq:frac_lap} and \eqref{eq:def_innerprod},
\begin{equation*} \label{eq:parts}
\langle (-\Delta)^s v, w \rangle = (v,w)_{H^s(\rn)}, \quad \forall v, w \in \tHs.
\end{equation*}

With the notation for fractional Sobolev norms introduced in \S\ref{Sobolev}, the variational form of problem \eqref{eq:dirichlet} reads:
\begin{equation} \label{eq:variational}
\mbox{find } u \in \tHs \mbox{ such that }  (u,v)_{H^s(\rn)} = \langle f, v \rangle \quad \forall v \in \tHs.
\end{equation}
We call $\|\cdot\|_{\tHs}$ the energy norm.

\subsection{Regularity of solutions}\label{regularity}

From this point on, we focus on the case $s \in (\frac12,1)$. 
In particular, $s$ has a fixed value from now on.
By definition, the solution $u$ to \eqref{eq:variational} belongs to $\widetilde{H}^s(\Omega)$. Furthermore, under the mild assumption of almost everywhere boundedness of the right-hand side, solutions belong to $u\in H^1_0(\Omega)$, with continuous dependence on the data.

\begin{pro}[$H^1$-estimate, see {\cite[Lemma 3.10]{AcBo17}}]
If $s \in (\frac12 , 1)$ and $f\in L^\infty(\Omega)$, then the solution $u$ of \eqref{eq:variational} belongs to $H^1_0({\Omega})$ and it satisfies
\[
|u|_{H^1(\Omega)} \lesssim \frac{ \| f\|_{L^\infty(\Omega)}}{2s-1}, 
\]
where the hidden constant depends on $\Omega$, but is uniformly bounded on $s \in (\frac12,1)$.
\end{pro}

A natural question is how much additional smoothness can be guaranteed under further assumptions on the data.  It is the case that, if the right-hand side $f$ possess certain H\"older regularity, then further regularity of $u$ follows.

\begin{pro}[Higher-order estimate, see {\cite[Theorem 3.5 and identity (3.6)]{BBNOS18}}] \label{pro:regularity}
Let $s \in \left(\frac12, 1\right)$ be given and $f\in C^\beta(\overline \Omega)$ for some $\beta>0$. Then, it holds that
\begin{equation}\label{continuous_dependence}
u \in \bigcap_{\epsilon>0}\widetilde{H}^{s+1/2-\epsilon}(\Omega),  \qquad \mbox{with} \quad  \|u\|_{\widetilde{H}^{s+1/2-\epsilon}(\Omega)} \lesssim \frac{\|f\|_{C^\beta(\overline \Omega)}}{\epsilon}, \quad \forall\epsilon \in(0,1/2).
\end{equation}
Furthermore, for $\beta \in (0, 2 - 2s)$, let $\ell \in (s+1/2, \beta + 2s)$ and $\alpha > \ell - s -1/2$. If $f\in C^\beta(\overline\Omega)$, then $u \in \widetilde{H}^\ell_\alpha (\Omega)$ and
\begin{equation}\label{eq:weighted_regularity}
|u|_{\widetilde H^\ell_\alpha (\Omega)} \lesssim \frac{\| f \|_{C^\beta(\overline\Omega)}}{(\beta + \ell - 2s) (1/2 + \alpha + s - \ell)} .
\end{equation}
The hidden constants depend on $\Omega$ and the dimension $n$.
\end{pro}

\begin{rmk}[Sharpness]\label{rmk_sharpness}
The first statement in the previous proposition is sharp. The boundary behavior \eqref{eq:boundary-grubb} causes, in general, solutions not to be in $H^{s+1/2}(\Omega)$. For instance, if $\Omega$ is a ball with center $x_0$ and radius $r$, and $f \equiv 1$, then
\[
u (x) = C \, ( r - |x-x_0|^2)^s_+.
\]
Additionally, interior regularity estimates for the fractional Laplacian are well understood, and indicate a lifting of order $2s$, measured either in the H\"older \cite{ROSe14} or in suitable Besov \cite{Co17} scales.
\end{rmk}

\begin{rmk}[Case of interest]
For a smooth right-hand side, a case of interest in \eqref{eq:weighted_regularity} to derive optimal approximation rates in the energy norm (see Subsection \ref{subseq:interp_estimates_graded_meshes}) is, for $\epsilon \in (0, 1/2)$,
\[
\beta = 1 - s, \quad \ell = 1 + s - 2\epsilon, \quad \alpha = 1/2 - \epsilon.
\]
This yields the estimate
\begin{equation}\label{eq:weighted_regularity_1/2}
\|u\|_{\widetilde{H}^{1+s-2\epsilon}_{1/2-\epsilon}(\Omega)}\lesssim \frac{\|f\|_{C^{1-s}(\overline \Omega)}}{\epsilon}, \quad \forall\epsilon \in (0,1/2).
\end{equation}
\end{rmk}

\subsection{Conforming approximations} \label{conforming_approx-1}
We consider {\em conforming} approximation of the fractional Laplace equation, realized with the help of globally continuous $P^1$ Lagrange finite elements on a {\em shape-regular} family of triangulations $({\mathcal T}_h)_h$ of $\Omega$ (see \cite[Definition 1.107]{ErGu04}); elements of triangulations are (closed) simplices of $\R^n$. We call $(\Vens_h)_h$ the discrete spaces, where $h$ denotes the mesh-size of a given triangulation; more precisely, we set 
\[
\Vens_h = \{ v \in C(\overline\Omega) \mbox{ s.t. } v |_T \in P^1 \ \forall T \in \mathcal{T}_h,  \ v |_{\pa\Omega} = 0 \} .
\]
Importantly, one has $\Vens_h\subset H^1_0(\Omega)$ for all $h$. We write $h_T$ for the diameter of an element $T \in \mathcal{T}_h$ (recall that $h=\max_T h_T$). In the following, given a set $\omega \subset \Omega$,  $S_\omega$ denotes the star of elements that intersect $\omega$, 
\[
S_\omega = \bigcup_{T' \colon \omega \cap T' \neq \emptyset} T' .
\]
Because elements are closed subsets of $\R^n$, $S_\omega$ is by definition a closed subset of $\R^n$. In particular, given $T \in \mathcal{T}_h$, we make use of the sets
\[
S_T = \bigcup_{T' \colon T \cap T' \neq \emptyset} T' \quad \mbox{and} \quad S_{S_T} = \bigcup_{T' \colon S_T \cap T' \neq \emptyset} T'.
\]
We set $u_h$ to be the solution of the discrete variational formulation
\begin{equation*} \label{eq:disc} 
\mbox{find $u_h \in \Vens_h$ such that } (u_h, v_h)_{H^s(\rn)} = \langle f, v_h \rangle \quad  \forall v_h \in \Vens_h.
\end{equation*}
It follows immediately that $u_h$ is the best approximation in $\Vens_h$ to the solution $u$ with respect to the energy norm:
\begin{equation}\label{eq:cea}
\| u - u_h \|_{\tHs} = \min_{v_h \in \Vens_h} \| u - v_h \|_{\tHs} .
\end{equation}

\subsection{Interpolation error} \label{conforming_approx-2}

From \eqref{eq:cea}, the only missing ingredient to deduce an a priori convergence rate (in the energy norm) for the fractional Laplace equation is an interpolation error estimate. This, combined with the regularity of solutions expressed in the first part of Proposition \ref{pro:regularity} gives the desired rate.

Let $I_h$ denote the Scott-Zhang interpolation operator \cite{ScZh90}. Local approximation estimates in integer-order norms are well-known,
\begin{equation} \label{eq:local_Scott-Zhang}
| v - I_h v |_{H^k(T)} \lesssim h_T^{\ell - k} | v |_{H^\ell(S_T)}, \quad \forall v \in H^\ell(\Omega), \  k \in \{0,1\} , \ \ell \in [k, 2].
\end{equation}
Moreover, it is a simple exercise to derive an approximation estimate in terms of the fractional weighted scale introduced in Definition \ref{def:weighted}. Indeed, it holds that
\begin{equation}\label{Scott-Zhang_weighted}
| v - I_h v |_{H^k(T)} \lesssim h_T^{\ell - k - \alpha} | v |_{H^\ell_\alpha(S_T)},  
\end{equation}
for all $v \in \widetilde H^\ell_\alpha(\Omega),$ $k \in \{0,1\},$ $\ell \in [k, 2],$ and $\alpha \in [0, \ell -k].$

Due to its non-local nature, in order to obtain a global interpolation estimate in a fractional-order norm, it is not desirable to have norms on elements on the left-hand side. However, such as developed in \cite{AcBo17}, it suffices to derive bounds over sets of the form $T \times S_T$, and use localization techniques \cite{Fa02}.

\begin{pro}[Local interpolation estimate; see \cite{Bort17,Ci13}]
Let $s \in (0,1)$ and $\ell \in [s, 2]$. Then, 
\begin{equation} \label{eq:interpolation_interior}
\int_T \int_{S_T} \frac{|(v-I_h v) (x) - (v-I_h v) (y)|^2}{|x-y|^{n+2s}} \, dy \, dx \lesssim h_T^{2(\ell-s)} |v|_{H^\ell({S_{S_T}})}^2 \quad \forall v \in \widetilde H^\ell(\Omega), 
\end{equation}
and, for $\alpha \in (0, \ell - s)$,
\begin{equation} \label{eq:interpolation_boundary}
\int_T \int_{S_T} \frac{|(v-I_h v) (x) - (v-I_h v) (y)|^2}{|x-y|^{n+2s}} \, dy \, dx \lesssim h_T^{2(\ell -s - \alpha)} | v |_{H^\ell_\alpha({S_{S_T}})}^2 \quad \forall v \in \widetilde H^\ell_\alpha(\Omega),
\end{equation}
with hidden constants that depend on $n$, $s$ and the shape-regularity of the meshes.
\end{pro}

\section{Quasi-uniform triangulations} \label{quasi-uniform}
Let $s \in \left(\frac12, 1\right)$ be given.
Throughout this section, we assume that the right-hand side $f$ belongs to $C^\beta(\overline\Omega)$ for some $\beta > 0$ and that approximations are performed on {\em quasi-uniform}  meshes \cite[Definition 1.140]{ErGu04}. In such a case, adding up the contributions on each patch of the form $T\times S_T$, and because of the a priori regularity of $u$  (recall \eqref{continuous_dependence}), we have the estimates
\begin{equation}\label{Scott-Zhang}
\left\{ \begin{aligned}
& \|u-I_hu\|_{\widetilde{H}^s(\Omega)} \lesssim \frac{h^{1/2-\epsilon}}{\epsilon} \|f\|_{C^\beta(\overline\Omega)}\\
& \|u-I_hu\|_{H^1(\Omega)} \lesssim \frac{h^{s-1/2-\epsilon}}{\epsilon} \|f\|_{C^\beta(\overline\Omega)}\\
\end{aligned}\right. \quad \forall\epsilon\in(0,1/2).
\end{equation}

Upon combining \eqref{eq:cea} and (\ref{Scott-Zhang}-top), it follows that the convergence rate of the finite element solutions towards the solution of the fractional Laplace problem in the energy norm is 
\begin{equation}\label{AcBo17_quasi-uniform}
\|u-u_h\|_{\widetilde{H}^s(\Omega)} \lesssim \frac{h^{1/2-\epsilon}}{\epsilon} \|f\|_{C^\beta(\overline\Omega)},\quad \forall\epsilon\in(0,1/2).
\end{equation}
Clearly, if $h$ is small enough, then taking $\epsilon = | \log h|^{-1}$ yields
\[
 \|u-u_h\|_{\widetilde{H}^s(\Omega)} \lesssim h^{1/2} | \log h | \|f\|_{C^\beta(\overline\Omega)} . 
 \]
In this section, we derive an error estimate in $H^1(\Omega)$-norm on quasi-uniform triangulations. For that purpose, we require an inverse inequality.

\begin{pro}[Inverse inequality] \label{pro:inverse_inequality}
Consider a sequence of discrete spaces $(\Vens_h)$ over quasi-uniform meshes. Then, it holds
\begin{equation}\label{Inverse_ineq}
 |v_h|_{H^1(\Omega)} \lesssim h^{s-1} \|v_h\|_{\widetilde H^s(\Omega)},\quad \forall h,\ \forall v_h\in\Vens_h.
\end{equation}
\end{pro}
\begin{proof}
It follows immediately by interpolation of the trivial identity $ |v_h|_{H^1(\Omega)} \le |v_h|_{H^1(\Omega)}$ and the standard global inverse inequality (for example, \cite[Corollary 1.141]{ErGu04})
\[
 |v_h|_{H^1(\Omega)} \le h^{-1} \|v_h\|_{L^2(\Omega)}.
\]
\end{proof}

From Proposition \ref{pro:inverse_inequality}, we infer a first bound on the error in the $H^1(\Omega)$-norm.

\begin{pro}[Convergence in $H^1(\Omega)$ on uniform meshes] \label{pro:H1_uniform}
Assume that $s \in (\frac12,1)$ and $f \in C^\beta(\overline\Omega)$ for some $\beta >0$. Consider a sequence of discrete spaces $(\Vens_h)$ over quasi-uniform meshes. Then, for $h$ sufficiently small, it holds
\[ 
\|u-u_h\|_{H^1(\Omega)} \lesssim h^{s-1/2} | \log h | \|f\|_{C^\beta(\overline \Omega)}. 
\]
\end{pro}
\begin{proof}
Let $\epsilon\in(0,1/2)$. In first place, using the triangle inequality and the interpolation estimate (\ref{Scott-Zhang}-bottom), we obtain
\[
\begin{aligned}
\|u-u_h\|_{H^1(\Omega)} &\le \|u-I_hu\|_{H^1(\Omega)} + \|I_hu-u_h\|_{H^1(\Omega)}\\
&\lesssim \frac{h^{s-1/2-\epsilon}}{\epsilon} \|f\|_{C^\beta(\overline \Omega)} + \|I_hu-u_h\|_{H^1(\Omega)} .
\end{aligned}
\]
Therefore, we need to bound $\|I_hu-u_h\|_{H^1(\Omega)}$. By the inverse inequality \eqref{Inverse_ineq} and using again the triangle inequality, it follows 
\[
\|I_hu-u_h\|_{H^1(\Omega)} \lesssim h^{s-1}\left( \|I_hu-u\|_{\widetilde H^s(\Omega)}+\|u-u_h\|_{\widetilde H^s(\Omega)} \right) .
\]
Finally, by bounding the right hand side above using (\ref{Scott-Zhang}-top) and 
\eqref{AcBo17_quasi-uniform}, we deduce
\[
\|u-u_h\|_{H^1(\Omega)} \lesssim \frac{h^{s-1/2-\epsilon}}{\epsilon} \|f\|_{C^\beta(\overline\Omega)}, \quad \forall \epsilon \in (0, 1/2).
\]
The proof is concluded upon setting $\epsilon = | \log h |^{-1}$ in the estimate above.
\end{proof}

For comparison with the results in the next section, we express the order of convergence in terms of the number of degrees of freedom. Since the meshes are quasi-uniform, ${\rm dim} \Vens_h \simeq h^{-n}$.

\begin{cor}[Complexity for uniform meshes] \label{cor:complexity}
With the same hypotheses as in Proposition \ref{pro:H1_uniform}, it holds
\[ 
\|u-u_h\|_{H^1(\Omega)} \lesssim \left( \rm{dim} \Vens_h \right)^{\frac{1/2-s}{n}} \log\left(\rm{dim} \Vens_h\right) \|f\|_{C^\beta(\overline\Omega)}. 
\]
\end{cor}

\section{Graded meshes} \label{graded}
The results from the preceding section establish that, given $s \in \left(\frac12,1\right)$, finite element solutions converge to the solution to \eqref{eq:dirichlet} in the $H^1(\Omega)$-norm. Nevertheless, the low regularity of the solution substantially affects the convergence rate.  
We recall that, according to Proposition~\ref{pro:regularity}, the regularity assumptions for the right hand side are quantified by $\beta$. This, in turn, determines the regularity of the solution (with the differentiability quantified by $\ell$ and the boundary degeneracy by $\alpha$). Finally, one has to take into account the $\widetilde{H}^s(\Omega)$-norm in which the error is measured.

Here we focus in two-dimensional problems, and exploit regularity in weighted fractional spaces by performing approximations on a sequence of suitably refined meshes. Since the solution is known to be more singular near the boundary of the domain, increased convergence rates are achieved by placing more degrees of freedom in that zone. More precisely, given a number $\mu \ge 1$ and a global mesh parameter $h$, we set the element diameters to be
\begin{equation} \label{eq:grading}
 h_T \simeq \left\lbrace
\begin{array}{ll}
 h^{\mu} &
\mbox{if } S_T\cap \partial \Omega\neq \emptyset, \\
h  \, d(T,\partial \Omega)^{(\mu-1)/\mu} & \mbox{otherwise.}
 \end{array}
 \right.
 \end{equation}

In definition \eqref{eq:grading}, considering $\mu = 1$ corresponds to uniform meshes, whereas for $\mu > 1$, elements become smaller as they approach $\pa \Omega$, which yields the so-called {\em graded} meshes. The mesh-size parameter $h$ has the intuitive interpretation of controlling the number of degrees of freedom as the mesh-size does for uniform meshes. Indeed, we have \cite{BoNoSa18}
\begin{equation} \label{eq:dofs}
\rm{dim} \Vens_h \simeq \left\lbrace \begin{array}{ll}
h^{-2}, & \mbox{ if } \mu \in [1,2), \\
h^{-2} |\log h| & \mbox{ if } \mu = 2, \\
h^{-\mu}  & \mbox{ if } \mu > 2. \\
\end{array} \right.
\end{equation}
As we shall see below, the optimal choice of $\mu$ depends on the parameters $s$, $\beta$, $\ell$ and $\alpha$. 

\begin{rmk}[Choice of $\mu$] \label{rmk:mu}
Estimate \eqref{eq:dofs} essentially says that, when grading according to \eqref{eq:grading}, considering the dimension of the resulting finite element space as a function of $\mu$, all increments in $\mu$ are ``for free'' as long as $\mu < 2$. When $\mu > 2$, there is an increment in the number of degrees of freedom with respect to $h$  that balances the expected gain due to the increase in differentiability. So, for smooth right-hand sides, optimal order of convergence is attained by imposing $\mu = 2$. Nevertheless, it may also be the case that the same order of convergence is attained by taking a lower $\mu$, which in turn would allow for less stringent hypotheses on $f$.
Keeping the grading as low as possible is of importance, for example, in order to avoid unnecessarily ill-conditioned systems. For the problems under consideration it is known \cite{AiMcTr99} that the finite element stiffness matrices ${\bf A}_h$ are conditioned according to $\kappa({\bf A}_h) \simeq ({\rm dim} \Vens_h)^{2s/n} \left( \frac{h_{max}}{h_{min}} \right)^{n-2s}$. Therefore, for two-dimensional problems, for meshes graded according to \eqref{eq:grading}, since $h_{max} \simeq h$ and $h_{min} \simeq h^\mu$, we deduce
\[
\kappa({\bf A}_h) \simeq ({\rm dim} \Vens_h)^{s} \, h^{(1-\mu)(2-2s)} \simeq 
\left\lbrace \begin{array}{ll}
h^{2-4s - \mu (2 - 2s)}, & \mbox{ if } \mu \in [1,2), \\
h^{-2} |\log h|^s & \mbox{ if } \mu = 2, \\
h^{2-2s-\mu(2-s)}  & \mbox{ if } \mu > 2. \\
\end{array} \right.
\]
\end{rmk}

\begin{rmk}[Problems in 1d or in 3d] \label{rmk:1d_3d}
Let us briefly consider the case of a one-dimensional, or of a three-dimensional, problem. For the one-dimensional case, it is easily checked that the counterpart of (\ref{eq:dofs}) may be written $\rm{dim}\Vens_h \simeq h^{-1}$ independently of $\mu$. Since $\mu$ can be taken as large as needed,  it is possible (computationally) to recover the optimal linear convergence order in the $H^1(\Omega)$ norm. See Remark \ref{rmk:1d_graded} and the experiments in \S\ref{sub:1d}.

On the other hand, for the three-dimensional case one can check that for graded meshes defined as in (\ref{eq:grading}), the counterpart of (\ref{eq:dofs}) now writes
\[
\rm{dim} \Vens_h \simeq \left\lbrace \begin{array}{ll}
h^{-3}, & \mbox{ if } \mu \in [1,3/2), \\
h^{-3} |\log h| & \mbox{ if } \mu = 3/2, \\
h^{-2\mu}  & \mbox{ if } \mu > 3/2. \\
\end{array} \right.
\]
This limits the control one may get with respect to $\mu$ to values in $[1,3/2)$, in constrast to $\mu\in[1, 2)$ for the two-dimensional case, and as a consequence, limits the order of convergence that can be obtained with this grading strategy in three-dimensional problems.
A natural cure for this problem, that stems from the anisotropic behavior of the solution near the boundary (cf. \eqref{eq:boundary-grubb}), is to use {\em anisotropic} meshes~\cite{Apel99}. 
\end{rmk}

\subsection{Interpolation estimates}\label{subseq:interp_estimates_graded_meshes}

Our first task is to bound a global interpolation error; naturally this is achieved by adding up local estimates. In view of the grading \eqref{eq:grading}, the key property is that, when summing up the local interpolation estimates for elements not touching $\partial \Omega$, the exponent in $d(T,\partial \Omega)$ is zero. This explicitly links the regularity of the function to be interpolated with the order of the norm in which we are measuring the error and with the grading parameter.

We illustrate the above discussion with an example: assuming that $f \in C^{\beta}(\overline \Omega)$ for some $\beta \in (0,2-2s)$, what is the minimal grading required to optimally bound --in the energy norm-- the interpolation error for the solution to \eqref{eq:dirichlet}? \\
Once we have set $\beta$ in the second part of Proposition \ref{pro:regularity}, we find that $u \in \widetilde H^\ell_\alpha (\Omega)$ for all
\begin{equation}\label{eq:constraints_on_l_alpha}
\ell \in (s + 1/2, \beta + 2s), \quad \alpha > \ell - s -1/2.
\end{equation}
Grading meshes according to \eqref{eq:grading}, from \eqref{eq:interpolation_interior} we deduce, for every $T$ such that $S_{S_T}\cap \partial \Omega = \emptyset$,
\[ \begin{aligned}
\int_T \int_{S_T} \frac{|(u-I_h u) (x) - (u-I_h u) (y)|^2}{|x-y|^{n+2s}} \, dy \, dx & \lesssim h_T^{2(\ell-s)} |u|_{H^\ell({S_{S_T}})}^2 \\
& \lesssim h^{2(\ell - s)} d(T, \pa\Omega)^{2 \frac{(\mu-1)(\ell - s)}{\mu}} | u |^2_{H^\ell({S_{S_T}})} . 
\end{aligned} \]
Observe that $\delta(x,y) \simeq d(T, \pa\Omega)$ for all $x,y \in S_T$ when $S_{S_T}\cap \partial \Omega = \emptyset$. Thus, in this case, we get
\[ \begin{aligned}
\int_T \int_{S_T} \frac{|(u-I_h u) (x) - (u-I_h u) (y)|^2}{|x-y|^{n+2s}} \, dy \, dx & \lesssim h^{2(\ell - s)} d(T, \pa\Omega)^{2 \left( \frac{(\mu-1)}{\mu} (\ell - s) - \alpha \right)} | u |^2_{H^\ell_\alpha({S_{S_T}})} . 
\end{aligned} \]
Imposing the exponent on $d(T,\partial \Omega)$ to be zero, we obtain the bound in energy norm
\begin{equation}\label{eq:energy_norm_for_graded_meshes}
\int_T \int_{S_T} \frac{|(u-I_h u) (x) - (u-I_h u) (y)|^2}{|x-y|^{n+2s}} \, dy \, dx
\lesssim h^{2(\ell - s)} | u |_{H^\ell_\alpha({S_{S_T}})}^2.
\end{equation}
Cancelling the exponent corresponds to choosing $\alpha$ equal to
\begin{equation} \label{eq:choice_alpha}
\alpha = (\ell - s) \frac{(\mu - 1)}{\mu}. 
\end{equation}
On the other hand, for every element $T$ such that $S_{S_T}\cap \partial \Omega \neq \emptyset$, we point out that this choice of the parameter $\alpha$ again yields the bound (\ref{eq:energy_norm_for_graded_meshes}), this time with the help of (\ref{eq:interpolation_boundary}). Summing up all contributions, we conclude that
\[ |v -I_h v |_{H^s(\Omega)} \lesssim h^{\ell - s} | v |_{\widetilde{H}^\ell_\alpha(\Omega)} \quad \forall v \in \widetilde{H}^\ell_\alpha(\Omega).\]
To realize (\ref{eq:constraints_on_l_alpha}) when setting $\alpha$ according to \eqref{eq:choice_alpha}, we are lead to the restriction
\[
\mu > 2 (\ell - s).
\]
Since we require $\mu \le 2$ but we also want to maximize $\ell-s$ (as this will be the resulting order of the interpolation error), it suffices to set
\begin{itemize}
\item for $\beta \in (0, 1-s]$: $\ell = \beta + 2s - 2\epsilon$, $\alpha = \beta + s - 1/2 - {\epsilon}$, $\mu = 2(\beta + s)$,  for some  $\epsilon \in (0, \beta + s - 1/2)$.
The resulting order is $h^{\beta + s - 2\epsilon}$.
\item for $\beta \in [1-s, 2-2s)$: $\ell = 1 + s - 2\epsilon$, $\alpha = 1/2 - {\epsilon}$, $\mu = 2$, for some $\epsilon \in (0, 1/2)$. The resulting order is $h^{1-2\epsilon}$.
\end{itemize}
\begin{rmk}[Optimal grading for energy norm] \label{rmk:grading_energy}
We remark that, in the case $\beta \in (0, 1-s]$, any other grading $\mu \in [2(\beta+s), 2]$ also delivers optimal interpolation rates. On the other hand, the interpolation estimate for $\beta = 1-s$, combined with \eqref{eq:weighted_regularity_1/2} and \eqref{eq:cea}, guarantees the linear (up to a logarithm) order of convergence  of the finite element approximations to \eqref{eq:dirichlet}. See \eqref{eq:Scott-Zhang_weightedv2s} and \eqref{eq:approximation_weighted} below.
\end{rmk}

A corollary of the previous discussion is that, for a fixed right hand-side $f$, the minimal grading to obtain optimal convergence estimates depends on the norm in which the error is measured. The next proposition further illustrates this point.

\begin{pro}[Interpolation error in $H^1(\Omega)$ over graded meshes] \label{pro:weighted_interpolation}
Let $\ell \in (1, 2]$ and $\alpha \in [0, \ell - 1)$.  Assume the meshes are constructed under the grading hypothesis \eqref{eq:grading} setting $\mu = \frac{\ell - 1}{\ell - 1 -\alpha}$ therein. Then, it holds that
\begin{equation} \label{eq:weighted_interpolation}
|v -I_h v |_{H^1(\Omega)} \lesssim h^{\ell - 1} | v |_{\widetilde{H}^\ell_\alpha(\Omega)} \quad \forall v \in \widetilde{H}^\ell_\alpha(\Omega).
\end{equation}
\end{pro}
\begin{proof}
We make use of the local interpolation identities \eqref{eq:local_Scott-Zhang} and \eqref{Scott-Zhang_weighted}. Indeed, if $S_T\cap \partial \Omega=\emptyset$, 
\[
| v - I_h v |_{H^1(T)}^2 \lesssim h_T^{2(\ell - 1)} | v |_{H^\ell(S_T)}^2  \lesssim  h^{2(\ell - 1)} d(T,\partial \Omega)^{\frac{2(\ell-1)(\mu-1)}{\mu}}  |v|^2_{H^\ell(S_T)} .
\]
Because $\delta(x,y) \simeq d(T,\partial \Omega)$ for all $x,y \in T$, 
we deduce
\[
| v - I_h v |_{H^1(T)}^2 \lesssim h^{2(\ell - 1)} d(T,\partial \Omega)^{\frac{2(\ell-1)(\mu-1)}{\mu} - 2 \alpha}  | v |^2_{H^\ell_\alpha(S_T)} .
\]
In order to make the exponent in the distance to the boundary term to vanish, we require that $\mu = \frac{\ell - 1}{\ell - 1 -\alpha}$, and conclude
\[
| v - I_h v |_{H^1(T)}^2 \lesssim h^{2(\ell - 1)} | v |^2_{H^\ell_\alpha(S_T)} \quad \mbox{if } S_T \cap \pa\Omega = \emptyset.
\]

On the other hand, if $S_T\cap \partial \Omega\ne\emptyset$, using our choice of $\mu$ we deduce
\[
| v - I_h v |_{H^1(T)}^2 \lesssim h_T^{2(\ell - 1 - \alpha)} | v |_{H^\ell_\alpha(S_T)}^2 = h^{2 \mu (\ell - 1 - \alpha)} | v |_{H^\ell_\alpha(S_T)}^2 = h^{2 (\ell - 1)} | v |_{H^\ell_\alpha(S_T)}^2 .
\]
The claim follows immediately.
\end{proof}

Combining the general interpolation estimate \eqref{eq:weighted_interpolation} with the regularity estimates from Proposition \ref{pro:regularity}, we optimally bound the interpolation error in $H^1(\Omega)$. We show that, with respect to the a priori estimates from Proposition \ref{pro:H1_uniform} and Corollary~\ref{cor:complexity}, it is possible to double the interpolation error rate by using graded meshes.

\begin{pro}[Interpolation of the solution] 
In problem \eqref{eq:dirichlet}, assume that $f \in C^\beta(\overline \Omega)$ for some $\beta>0$ and that triangulations are constructed according to \eqref{eq:grading} with $\mu = 2$. Then, for the Scott-Zhang interpolation operator $I_h$,
\begin{equation} \label{Scott-Zhang_weighted_v2}
| u - I_h u |_{H^1(\Omega)} \lesssim h^{2(s-1/2-\epsilon)} | u |_{H^{2s - 2\epsilon}_{s - 1/2 - \epsilon}(\Omega)} \lesssim  \frac{h^{2(s-1/2-\epsilon)}}{\epsilon} \|f\|_{C^\beta(\overline\Omega)},\quad \forall\epsilon \in (0,\beta/2).
\end{equation}
In terms of degrees of freedom, for sufficiently refined meshes, the estimate above reads
\[
| u - I_h u |_{H^1(\Omega)} \lesssim ( {\rm dim} \Vens_h )^{1/2-s} \log ({\rm dim} \Vens_h ) \|f\|_{C^\beta(\overline\Omega)}.
\]
\end{pro}
\begin{proof}
From the second part of Proposition \ref{pro:regularity}, we know that $u \in \widetilde H^\ell_\alpha (\Omega)$ for all $\ell \in (s + 1/2, \beta + 2s)$ and $\alpha > \ell - s - 1/2$.  
Thus, given $\epsilon$ sufficiently small, we set $\ell = 2s - 2\epsilon$; we remark that choosing $\alpha = \frac{\ell - 1}{2}$ satisfies the restriction for this parameter, and yields $\mu = 2 = \frac{\ell - 1}{\ell - 1 - \alpha}$. 

Therefore, the first inequality in \eqref{Scott-Zhang_weighted_v2} follows from Proposition \ref{pro:weighted_interpolation}. The second inequality is a consequence of the regularity estimate \eqref{eq:weighted_regularity}.
\end{proof}

\begin{rmk}[Higher regularity assumptions] A question in order is whether the order of the interpolation error can be increased if we demand more regularity on the right-hand side $f$. For example, let us assume that $f \in C^2(\overline\Omega)$, so that we can take $\ell = 2 - 2\epsilon$ and $\alpha > 3/2 - s - 2\epsilon$ in \eqref{eq:weighted_regularity}, so that
\[
|u|_{\widetilde H^{2-2\epsilon}_{\alpha}(\Omega)} \lesssim \frac{\| f \|_{C^2(\overline\Omega)}}{\alpha - (3/2 - s - 2\epsilon)}.
\]
The same computations as in the proof of Proposition \ref{pro:weighted_interpolation} show that, to maximize the interpolation order (in $h$), the grading parameter should be chosen as
\[
\mu = \frac{\ell - 1}{\ell - 1 -\alpha} >  \frac{1 - 2 \epsilon}{s - 1/2} > 2. 
\]
Therefore, even though the interpolation error (in the $H^1(\Omega)$ norm) is of the order of $h^{\ell-1} = h^{1 -2\epsilon}$, in terms of degrees of freedom we have
\[
| u - I_h u |_{H^1(\Omega)} \lesssim ( {\rm dim} \Vens_h )^{1/2-s} \log ({\rm dim} \Vens_h ) \|f\|_{C^2(\overline\Omega)}.
\]
Having assumed higher regularity from $f$ has lead to no gain: the order is the same as in \eqref{Scott-Zhang_weighted_v2}. Also, it should be noted that, as described in Remark \ref{rmk:mu}, a more severe grading negatively affects the conditioning of the resulting system.
\end{rmk}

\subsection{Global inverse inequality}

Our next task is to derive an adequate inverse inequality for discrete functions over graded meshes. The non-uniformity of the meshes substantially affects the order (with respect to $h$) of such an estimate. In spite of its pessimistic character, the following proposition is instrumental to derive convergence rates in the $H^1(\Omega)$-norm.

\begin{pro}[Inverse inequality on graded meshes]
Consider a sequence of discrete spaces $(\Vens_h)$ over a sequence of meshes constructed according to \eqref{eq:grading} with a grading parameter $\mu$. Then, it holds
\begin{equation}\label{eq:inverse_graded}
 |v_h|_{H^1(\Omega)} \lesssim h^{\mu(s-1)} \|v_h\|_{\widetilde H^s(\Omega)},\quad \forall h,\ \forall v_h\in\Vens_h.
\end{equation}
\end{pro}
\begin{proof}
As in Proposition \ref{pro:inverse_inequality}, the proof follows by interpolation. In view of \eqref{eq:grading}, the local inverse inequality
\[
| v_h |_{H^1(T)} \lesssim h_T^{-1} \| v_h \|_{L^2(T)}, \quad T \in \mathcal{T}_h,
\]
can be written as
\[
| v_h |_{H^1(T)} \lesssim 
\left\lbrace
\begin{array}{ll}
 h^{-\mu} \| v_h \|_{L^2(T)} &
\mbox{if } S_T\cap \partial \Omega\neq \emptyset, \\
h^{-1} d(T, \pa\Omega)^{- (\mu-1)/\mu} \| v_h \|_{L^2(T)} & \mbox{ if }  S_T\cap \partial \Omega =\emptyset. \\
 \end{array}
 \right.
\]
Since $d(T, \pa\Omega) \gtrsim h^\mu$ for all elements $T$ such that $S_Ts\cap \partial \Omega =\emptyset$, we obtain the global inverse inequality
\[
| v_h |_{H^1(\Omega)} \lesssim h^{-\mu} \| v_h \|_{L^2(\Omega)}.
\]
By interpolation, we conclude \eqref{eq:inverse_graded}.
\end{proof}

\subsection{Convergence in $H^1(\Omega)$}

We are finally in position to derive a convergence rate for the solution to \eqref{eq:dirichlet} in $H^1(\Omega)$ using graded meshes. For that purpose recall, from Remark \ref{rmk:grading_energy}, that if $f \in C^{1-s}(\overline\Omega)$, then considering the Scott-Zhang interpolation on meshes graded according to \eqref{eq:grading} with $\mu=2$, we have
\begin{equation} \label{eq:Scott-Zhang_weightedv2s}
\| u - I_h u \|_{\widetilde H^s(\Omega)} \lesssim \frac{h^{1 - 2\epsilon}}{\epsilon} \| f \|_{C^{1-s}(\overline\Omega)}.
\end{equation}
This, combined with the best approximation property \eqref{eq:cea}, gives
\begin{equation} \label{eq:approximation_weighted}
\| u - u_h \|_{\widetilde H^s(\Omega)} \lesssim \frac{h^{1 - 2\epsilon}}{\epsilon} \| f \|_{C^{1-s}(\overline\Omega)}
\end{equation}

\begin{pro}[Convergence in $H^1(\Omega)$ on graded meshes] \label{pro:H1_graded}
Assume that $s \in (\frac12,1)$ and $f \in C^{1-s}(\overline\Omega)$. Consider a sequence of discrete spaces $(\Vens_h)$ over meshes graded according to \eqref{eq:grading} with $\mu =2$. Then, for $h$ sufficiently small, it holds
\[ 
\|u-u_h\|_{H^1(\Omega)} \lesssim h^{2(s-1/2)} | \log h | \|f\|_{C^{1-s}(\overline\Omega)}. 
\]
In terms of the dimension of the discrete spaces, the estimate above reads
\[ 
\|u-u_h\|_{H^1(\Omega)}  \lesssim ( {\rm dim} \Vens_h )^{1/2-s} \log ({\rm dim} \Vens_h ) \|f\|_{C^{1-s}(\overline\Omega)}. 
\]
\end{pro}
\begin{proof}
The proof follows the steps from Proposition \ref{pro:H1_uniform}, but replacing (\ref{Scott-Zhang}-bottom) by \eqref{Scott-Zhang_weighted_v2}.
For $\epsilon\in(0,1/2)$, using the triangle inequality and the interpolation estimate \eqref{Scott-Zhang_weighted_v2}, we obtain
\[
\begin{aligned}
\|u-u_h\|_{H^1(\Omega)} &\le \|u-I_hu\|_{H^1(\Omega)} + \|I_hu-u_h\|_{H^1(\Omega)}\\
&\lesssim \frac{h^{2(s-1/2-\epsilon)}}{\epsilon} \|f\|_{C^{1-s}(\overline \Omega)} + \|I_hu-u_h\|_{H^1(\Omega)} .
\end{aligned}
\]
Therefore, we need to bound $\|I_hu-u_h\|_{H^1(\Omega)}$. By the inverse inequality \eqref{eq:inverse_graded} and using again the triangle inequality, it follows 
\[
\|I_hu-u_h\|_{H^1(\Omega)} \lesssim h^{2(s-1)}\left( \|I_hu-u\|_{\widetilde H^s(\Omega)}+\|u-u_h\|_{\widetilde H^s(\Omega)} \right) .
\]
Finally, we use \eqref{eq:Scott-Zhang_weightedv2s} and \eqref{eq:approximation_weighted} to bound the right hand side above and deduce that 
\[
\|u-u_h\|_{H^1(\Omega)} \lesssim \frac{h^{2(s-1/2-\epsilon)}}{\epsilon} \|f\|_{C^{1-s}(\overline \Omega)}, \quad \forall \epsilon \in (0, 1/2).
\]
Setting $\epsilon = | \log h |^{-1}$ in this inequality, we conclude the proof of the first statement. The second part of the proposition follows by identity \eqref{eq:dofs}.
\end{proof}

\begin{rmk}[Error estimates in 1d using graded meshes] \label{rmk:1d_graded} 
As we pointed out in Remark~\ref{rmk:1d_3d}, for one-dimensional problems, it is possible to arbitrarily increase the grading parameter $\mu$ without affecting the relation $\rm{dim}\Vens_h \simeq h^{-1}$. When considering error estimates in the energy norm, this allows to obtain convergence with order $2-s$ by taking $\mu = 1/(s-1/2) > 2$.

On the other hand, it is clear that a large $\mu$ affects the inverse inequality \eqref{eq:inverse_graded}, and limits the theoretical order of convergence in the $H^1(\Omega)$ norm. Indeed, a direct calculation shows that the optimal error estimate that can be obtained as in Proposition \ref{pro:H1_graded} is given by taking $\mu = 2(2-s)$:
\begin{equation} \label{eq:1d_estimate}
\|u-u_h\|_{H^1(\Omega)} \lesssim h^{2(s-1/2)(2-s)} | \log h | \|f\|_{C^{2-2s}(\overline\Omega)}. 
\end{equation}
In \S\ref{sub:1d} we perform experiments that illustrate the sharpness of this estimate.
\end{rmk}

\section{Numerical experiments} \label{num_exp}

In this section, we display some results for problems in one- and two-dimensional domains, both for uniform and graded meshes. The outcomes of our numerical experiments matches the prediction that the convergence rates deteriorate as $s \to 1/2$. For completeness, we include the negative results for the limit case $s=1/2$ in order to further illustrate the fact that the solution to \eqref{eq:dirichlet} may not belong to $H^1(\Omega)$ (see Remark~\ref{rmk_sharpness}).

Unless $\Omega$ is a ball, it is not possible to derive closed expressions for the solution $u$ to \eqref{eq:dirichlet}; thus, we restrict the numerical examples to such domains.  Specifically, consider the Jacobi polynomials $P_k^{(\alpha, \beta)} \colon [-1,1] \to \R,$ given by 
{\small{
\[
P_k^{(\alpha, \beta)}(z) =  \frac{\Gamma (\alpha+k+1)}{k!\,\Gamma (\alpha+\beta+k+1)} 
\sum_{m=0}^k {k\choose m} \frac{\Gamma (\alpha + \beta + k + m + 1)}{\Gamma (\alpha + m + 1)} 
\left(\frac{z-1}{2}\right)^m,
\]
}}
 and the weight function  $\omega^s:\R^n\to \R,$
\[ \omega^s(x) = (1 - |x|^2)_+ ^s. \]
Then, given $k \in \mathbb{N}$, $s \in (0,1)$, and the right-hand side
\begin{equation} \label{eq:right_hand_side}
f (x) =   P_k^{(s, \, n/2-1)} ( 2 |x|^2 - 1),
\end{equation}
the solution to \eqref{eq:dirichlet} in the unit ball $B(0,1)\subset \R^n$ is  \cite[Theorem 3]{DyKuKw17}
\begin{equation} \label{eq:solution}
u (x) = \frac{k! \, \Gamma\left(\frac n2+k\right)}{2^{2s} \, \Gamma(1+s+k) \Gamma\left(\frac n2+s+k\right)} \, \omega^s(x) \, P_k^{(s, \, n/2-1)} ( 2 |x|^2 - 1).
\end{equation}

\subsection{One-dimensional problems with constant right hand side} \label{sub:1d}

As a first example, we take $\Omega = (-1,1)$ and $f = 1$. Then, according to \eqref{eq:solution}, for $s\in(0,1)$, the solution to \eqref{eq:dirichlet}  is given by
\[
u(x) = \frac{\sqrt{\pi}}{2^{2s}\Gamma(1+s)\Gamma(1/2+s)} (1 - x^2)^s_+ .
\]
We compute finite element solutions on meshes with~$N \in \{1000, 2000, \ldots , 10000 \}$ equally spaced nodes and the corresponding errors in the $H^1(\Omega)$ norm for $s \in \{ 0.6, 0.7, 0.8, 0.9\}$.
We display our results in Figure \ref{fig:errs_interval}. These are in good agreement with the estimates from Proposition \ref{pro:H1_uniform}.

\begin{figure}[ht]
	\centering
	\includegraphics[width=1\textwidth]{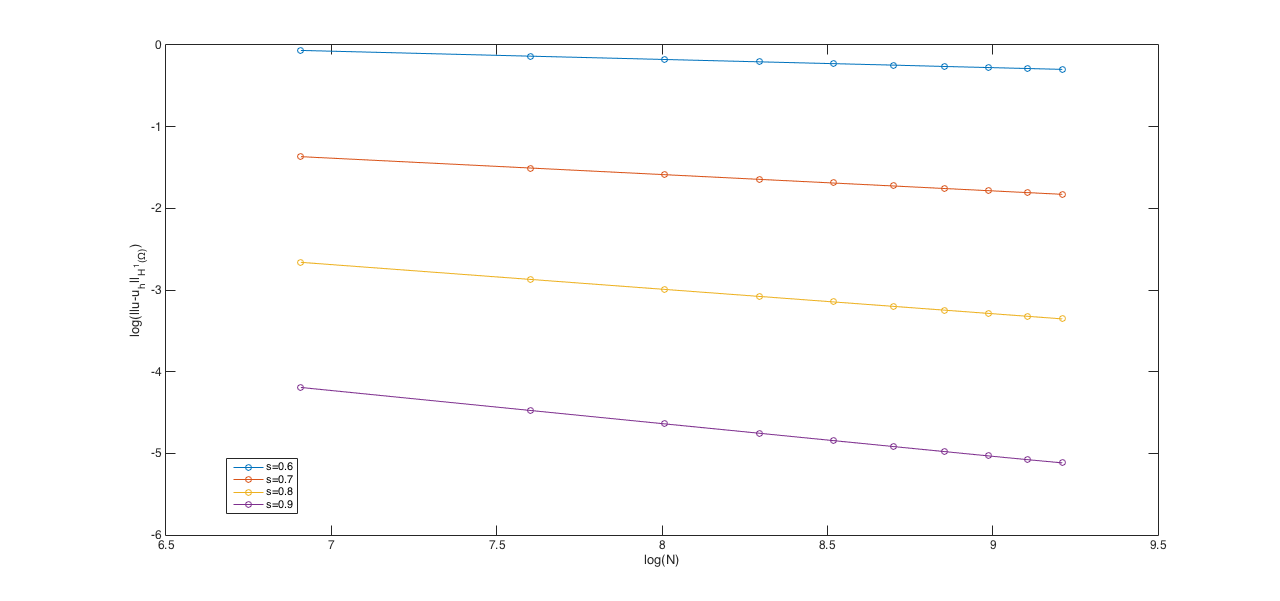}
	\caption{Errors for the first example described in \S\ref{sub:1d}. Least squares fitting of the data yields estimated orders of convergence $0.101$ for $s=0.6$, $0.200$ for $s=0.7$,  $0.301$ for $s=0.8$, and $0.402$ for $s=0.9$.}
	\label{fig:errs_interval}
\end{figure}

Moreover, we run the same experiment for $s=0.5$. Naturally, in this case the solution $u$ does not belong to $H^1(\Omega)$. Therefore, we just compute the $H^1(\Omega)$ seminorm of the discrete solutions; Figure \ref{fig:interval_05} gives evidence that these are indeed unbounded. 

\begin{figure}[ht]
	\centering
	\includegraphics[width=1\textwidth]{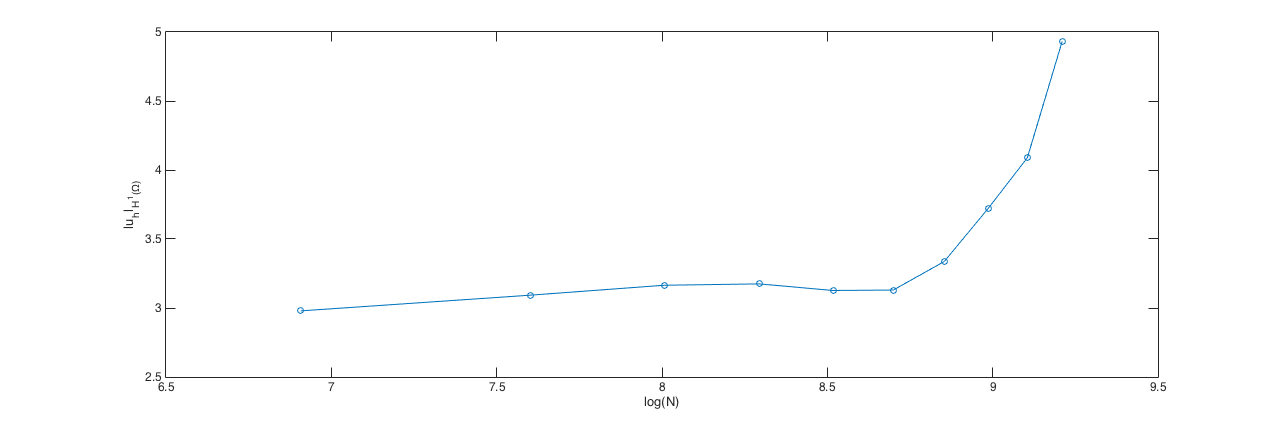}
	\caption{$H^1(\Omega)$ seminorm of the finite element solutions for $s=0.5$ as a function of the number of degrees of freedom.}
	\label{fig:interval_05}
\end{figure}

As a second example in one dimension, we build graded meshes using either $\mu_1 = 2 (2-s)$ or $\mu_2 = \frac{1}{s-1/2}$. As described in Remark \ref{rmk:1d_graded}, convergence with order $2(s-1/2)(2-s)$ can be obtained grading meshes according to $\mu_1$. As for $\mu_2$, although we cannot apply the argument from Proposition \ref{pro:H1_graded}, Table \ref{tab:1d_graded} shows that experimentally we recover linear convergence rates in the $H^1(\Omega)$ norm. We point out that, especially for $\mu_2$ with $s$ near $1/2$, the large magnitude of the required grading yields very small elements near the boundary, and therefore limits the number of nodes that the meshes can have before reaching machine precision. In these sets of experiments, for every $s$ we considered four meshes with the number of nodes that guaranteed that the smallest elements were closest to being of size $\{ 10^{-6} , \ldots, 10^{-9} \}$.

\begin{table}
\label{tab:1d_graded}  
\begin{tabular}{ccccc}
\hline\noalign{\smallskip}
$s$ & $\mu_1 = 2 (2-s)$ & Computed order ($\mu_1$) & $\mu_2 = \frac{1}{s-1/2}$ & Computed order ($\mu_2$) \\
\noalign{\smallskip}\hline\noalign{\smallskip}
$0.6$ &  $2.8$ & $0.29 \quad (0.28)$  &  $10$ & $1.00$ \\
$0.7$ &  $2.6$ & $0.53 \quad (0.55)$ &  $5$ & $0.95$ \\
$0.8$ &  $2.4$ & $0.74 \quad (0.72)$ &  $10/3$ & $0.97$ \\
$0.9$ &  $2.2$ & $0.93 \quad (0.88)$ &  $2.5$ & $0.99$ \\
\noalign{\smallskip}\hline \\ 
\end{tabular}
\caption{Observed convergence rates in the $H^1(\Omega)$ norm for the one-dimensional homogeneous Dirichlet problem using graded meshes. In the column with the computed order using $\mu_1$, the predicted order $2(s-1/2)(2-s)$ is in parenthesis.}
\end{table}

\subsection{Two-dimensional problems} \label{sub:2d}

We now turn our attention to problems posed in the two-dimensional unit ball $\Omega = B(0,1) \subset \mathbb{R}^2$. In first place, we set $k=0$ in \eqref{eq:right_hand_side} and consider problems with $s \in \{0.6,0.7,0.8,0.9\}$. With the aid of the code from \cite{AcBeBo17}, we compute solutions using both uniform and graded meshes, with $\mu = 2$. Table \ref{tab:2d_f_constant} summarizes our findings, which are in accordance with the theory: in all cases and with respect to $\rm{dim}\Vens_h$, the observed order of convergence employing uniform meshes is about $\frac{s-1/2}{2}$ (cf. Corollary~\ref{cor:complexity}), while this order is doubled when taking graded meshes (cf. second part of Proposition \ref{pro:H1_graded}).

\begin{table}[h]
\label{tab:2d_f_constant}
\begin{tabular}{ccc}
\hline\noalign{\smallskip}
$s$ & Computed order (uniform) & Computed order (graded) \\
\noalign{\smallskip}\hline\noalign{\smallskip}
$0.6$ &  $0.04 \quad (0.05)$ & $0.08 \quad (0.10)$ \\
$0.7$ &  $0.08 \quad (0.10)$ & $0.18 \quad (0.20)$ \\
$0.8$ &  $0.13 \quad (0.15)$ & $0.30 \quad (0.30)$ \\
$0.9$ &  $0.19 \quad (0.20)$ & $0.41 \quad (0.40)$ \\
\noalign{\smallskip}\hline \\ 
\end{tabular}
\caption{Observed convergence rates in the $H^1(\Omega)$ norm for the two-dimensional homogeneous Dirichlet problem with constant right-hand side. The orders predicted by either Corollary~\ref{cor:complexity} and Proposition \ref{pro:H1_graded} are in parenthesis.}
\end{table}

Figure \ref{fig:gradients_2d} exhibits the logarithm of the norm of the broken gradient of discrete solutions for $s=0.6$ over certain uniform and graded ($\mu=2$) meshes with about the same number of degrees of freedom. We point out that, in this example, the exact solution verifies $|\nabla u(x)| \sim (1-|x|^2)^{s-1}$ for $|x|\sim 1$. The better capability of the graded mesh to capture the singularity of the gradient at the boundary of the domain is apparent.

\begin{figure}[ht]
	\centering
	\includegraphics[width=0.9\textwidth]{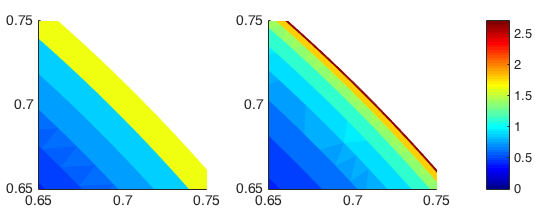}
	\caption{Logarithm of the norm of the gradient of the discrete solutions to the first example in \S\ref{sub:2d} for $s=0.6$ using uniform (left) and graded (right) meshes with approximately the same number of degrees of freedom (12636 and 12656, respectively). The pictures correspond to a zoom on the square $[0.65, 0.75]^2$.}
	\label{fig:gradients_2d}
\end{figure}

As a final illustration, we consider a problem with non-constant right-hand side. Setting $k=1$ in \eqref{eq:right_hand_side}, we obtain that
\[
u (x) =  \frac{1}{2^{2s} (\Gamma(2+s))^2} \, \left(1-|x|^2\right)^s_+ \, \left((2+s)|x|^2 - 1 \right)  
\]
solves \eqref{eq:dirichlet} in $B(0,1) \subset \R^2$ for
\[
f(x) = (2+s) |x|^2 - 1 .
\]
We compute solutions over meshes graded according to $\mu = 2$, and summarize our findings in Figure \ref{fig:errs_ball}. These are in good agreement with the orders $s-1/2$, with respect to ${\rm dim}\Vens_h$, predicted by Proposition \ref{pro:H1_graded}.

\begin{figure}[ht]
	\centering
	\includegraphics[width=1\textwidth]{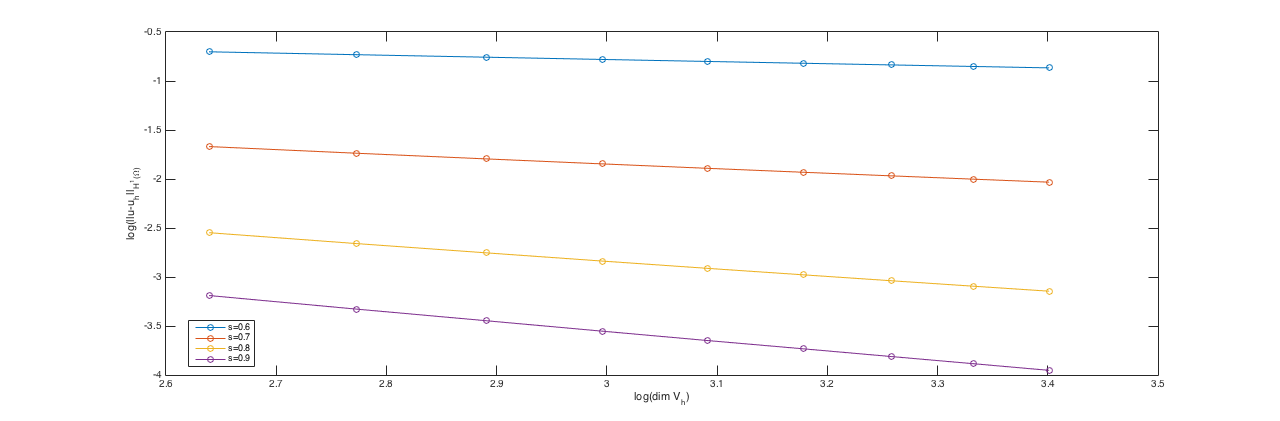}
	\caption{Errors for the example with non-constant right hand side in the unit ball in $\R^2$ with respect to ${\rm dim} \Vens_h$. Least squares fitting of the data yields estimated orders of convergence $0.09$ for $s=0.6$, $0.20$ for $s=0.7$,  $0.33$ for $s=0.8$, and $0.42$ for $s=0.9$.}
	\label{fig:errs_ball}
\end{figure}

\section{Concluding remarks} \label{sec:conclusion}
In this paper, we analyzed convergence rates for finite element discretizations of the integral fractional Laplacian over bounded domains. We showed that the a priori convergence rates can be improved by resorting to graded meshes.

For the sake of clarity, we restricted the discussion to the $H^1(\Omega)$-norm; nevertheless, the arguments presented here can be applied to obtain convergence rates in $H^t(\Omega)$ for all $t \in (s,s+1/2)$. For instance, the claim in Proposition \ref{pro:H1_uniform} can be extended to
\[
\| u - u_h \|_{H^t(\Omega)} \lesssim h^{s+1/2-t} | \log h | \| f \|_{C^\beta(\overline\Omega)}, \quad t \in (s,s+1/2).
\]
Analogous estimates can be obtained for discretizations on graded meshes. In such a case, the optimal grading depends on the regularity of the data and the norm in which the error is measured.

The class of graded meshes we considered allow to deliver optimal convergence rates in one-dimensional domains. However, in two and three dimensions, in spite of accelerating the convergence of the finite element approximations, such meshes are not capable of delivering optimal convergence rates. Shape-regularity limits the grading parameter that can be taken while keeping control of the number of degrees of freedom. Therefore, discretizations using anisotropic elements are required. To the best of the authors' knowledge, there is no interpolation theory using anisotropic fractional-order Sobolev spaces in the literature.

\bibliographystyle{plain}
\bibliography{BoCi17_Nonlocal}

\end{document}